\documentclass{amsart}
\usepackage{geometry}
\usepackage[utf8]{inputenc}
\usepackage[english,shorthands=off]{babel}
\usepackage[OT1]{fontenc}
\usepackage{amsmath}
\usepackage{amsfonts}
\usepackage{amssymb}
\usepackage{amsthm}
\usepackage{amscd}
\usepackage{mathabx}
\usepackage{quiver}
\usepackage{xcolor}
\usepackage{xfrac}
\usepackage{mathrsfs}
\usepackage{graphicx}
\usepackage[all,arrow,color,cmtip]{xy}
\usepackage{tikz-cd}
\usetikzlibrary{cd}
\usepackage{extarrows}
\usepackage{hyperref}
 \hypersetup{colorlinks,
 citecolor=black,
 filecolor=black,
 linkcolor=black,
 urlcolor=black,
 }

 \theoremstyle{plain}
\newtheorem{teo}{Theorem}[section]
\newtheorem{prop}[teo]{Proposition}
\newtheorem{lema}[teo]{Lemma}
\newtheorem{coro}[teo]{Corollary}

 \theoremstyle{definition}
 \newtheorem{defi}[teo]{Definition}
\newtheorem{obs}[teo]{Remark}
\newtheorem{ej}[teo]{Example}

\expandafter\let\expandafter\oldproof\csname\string\proof\endcsname
\let\oldendproof\endproof
\renewenvironment{proof}[1][Proof]{\oldproof[\normalfont \bfseries #1]} 
{\oldendproof}

\begin{document}

\title[Singularity Categories]{Recollements of Triangulated Categories and the Singularity Category of a Triangular Matrix Category.}



\author{Juan Andr\'es Orozco Guti\'errez}
\address{Departamento de Matem\'aticas, Facultad de Ciencias, Universidad Nacional Aut\'onoma de M\'exico,
Circuito Exterior, Ciudad Universitaria, C.P. 04510, Ciudad de M\'exico, MEXICO.}
\curraddr{}
\email{juan\_andres@ciencias.unam.mx}
\thanks{}

\author{Valente Santiago Vargas}
\address{Departamento de Matem\'aticas, Facultad de Ciencias, Universidad Nacional Aut\'onoma de M\'exico,
Circuito Exterior, Ciudad Universitaria, C.P. 04510, Ciudad de M\'exico, MEXICO.}
\curraddr{}
\email{valente.santiago@ciencias.unam.mx}
\thanks{The authors thank to the project PAPIIT-Universidad Nacional Aut\'onoma de M\'exico IN100124.}


\date{}

\dedicatory{}

\commby{}

\begin{abstract}

Introduced by Buchweitz, the singularity category of an algebra $A$ measures its homological singularity, vanishing if and only if $A$ has finite global dimension. This notion extends naturally to the context of $k$-categories. In this paper we study the singularity category of a triangular matrix category $\Lambda:=\left[ \begin{smallmatrix}
\mathcal{T} & 0 \\ 
M & \mathcal{U}
\end{smallmatrix}\right]$. By utilizing the framework of recollements, we provide a characterization of this category, proving that when certain homological conditions are satisfied, there exists an equivalence of singularity categories $D_{sg}(\mathrm{Mod(\Lambda)})\simeq D_{sg}(\mathrm{Mod}(\mathcal{U}))$. This result generalizes the one obtained by Pin Liu and Ming Lu in \cite{liu}. 
 \end{abstract}

\maketitle

\tableofcontents

\section{Introduction}

The singularity category $\mathbf{D}_{sg}(A)$ of an algebra A over a field $k$, introduced by
R.O. Buchweitz in \cite{Buchweitz}, is defined as the Verdier quotient 
$$\mathbf{D}_{sg}(A)=\mathbf{D}^{b}(\mathrm{mod}(A))/\mathrm{perf}(A)$$ of the bounded derived category $\mathbf{D}^{b}(\mathrm{mod}(A))$ by the category of perfect complexes.
In recent years, D. Orlov (\cite{Orlov1} ) rediscovered the notion of singularity categories in his
study of B-branes on Landau-Ginzburg models in the framework of the Homological
Mirror Symmetry Conjecture. The singularity category measures the homological
singularity of an algebra in the sense that an algebra A has finite global dimension
if and only if its singularity category $\mathbf{D}_{sg}(A)$ vanishes.\\
A celebrated theorem of Buchweitz (see \cite{Buchweitz}) shows that if $R$ is an Iwanaga-Gorenstein ring, then the stable
category of Cohen-Macaulay R-modules is triangle equivalent to the singularity category of $R$.\\
Let $\mathcal{C}$ be an additive category. We denote by $\mathrm{Mod}(\mathcal{C})$ the category of left $\mathcal{C}$-modules. The notion of singular equivalences for rings is further extended to
additive categories $\mathcal{C}$ by using $\mathrm{Mod}(\mathcal{C})$ as follows: We take the Verdier quotient
$$\mathbf{D}_{sg}(\mathrm{Mod}(\mathcal{C}))=\mathbf{D}^{b}(\mathrm{Mod}(\mathcal{C}))/\mathrm{Perf}(\mathrm{Mod}(\mathcal{C}))$$ and we call this the singularity category of $\mathcal{C}$. We say that two additive categories $\mathcal{C}$ and $\mathcal{C'}$ are singularly equivalent if there exists a triangle equivalence
$$\mathbf{D}_{sg}(\mathrm{Mod}(\mathcal{C}))\simeq \mathbf{D}_{sg}(\mathrm{Mod}(\mathcal{C'})).$$
Then one natural question is: when two additive categories are singularly equivalent? In general this is a difficult question.\\ The main purpose of this paper is to explore this question for triangular matrix categories. Let $\Lambda:=\left[ \begin{smallmatrix}
\mathcal{T} & 0 \\ 
M & \mathcal{U}
\end{smallmatrix}\right]$ be a triangular matrix category. We show that if $pd_{\mathcal{U}}M<\infty$, $pd(M_\mathcal{T})<\infty$ and $gl.dim(\mathcal{T})<\infty$,  then there exists a singular equivalence between $\Lambda$ and $\mathcal{U}$ (see Corollary \ref{main}). This result is a generalization of the result obtained in \cite{liu} by Pin Liu and Ming Lu. We give an explicit example in the context of representation of infinite quivers.

\section{Preliminaries}
Throughout this paper we will consider small $k$-categories  $\mathcal{C}$ over a field $k$, which means that the class of objects of $\mathcal{C}$ forms a set, the sets $\mathrm{Hom}_{\mathcal{C}}(X,Y)$ are $k$-vector spaces and the composition is $k$-bilinear.  For conciseness, we will sometimes write $\mathcal{C}(X,Y)$ instead of $\mathrm{Hom}_{\mathcal{C}}(X,Y)$. Furthermore, we refer to \cite{Mitchelring} for basic properties of $k$-categories.\\
Let $\mathcal{C}$ and $\mathcal{D}$ be $k$-categories. A $k$-functor is an aditive functor $F:\mathcal{C}\rightarrow \mathcal{D}$ such that $F:\mathcal{C}(X,Y)\rightarrow \mathcal{D}(F(X),F(Y))$ is a $k$-linear transformation.
For $k$-categories $\mathcal{C}$ and  $\mathcal{D}$, we consider the category of all $k$-functors from $\mathcal{C}$ to $\mathcal{D}$, which we denote by $\mathrm{Fun}_{k}(\mathcal{C},\mathcal{D})$. Given an arbitrary small preadditive category $\mathcal{C}$, the category of all additive functors  $\mathrm{Fun}_{\mathbb{Z}}(\mathcal{C},\mathbf{Ab})$ is denoted by $\mathrm{Mod}(\mathcal{C})$ and is called the category of left $\mathcal{C}$-modules. Sometimes we write $_{\mathcal{C}}M$ when $M\in \mathrm{Mod}(\mathcal{C})$. When
 $\mathcal{C}$ is a $k$-category, there is an isomorphism of categories $\mathrm{Fun}_{\mathbb{Z}}(\mathcal{C},\mathbf{Ab})\cong \mathrm{Fun}_{k}(\mathcal{C},\mathrm{Mod}(k))$ where $\mathrm{Mod}(k)$ denotes the category of $k$-vector spaces. Thus, we can identify $\mathrm{Mod}(\mathcal{C})$ with $\mathrm{Fun}_{k}(\mathcal{C},\mathrm{Mod}(k))$.  If $\mathcal{C}$ is a $k$-category, we consider its opposite  category $\mathcal{C}^{op}$, which is also a $k$-category; and we construct the category of right $\mathcal{C}$-modules $\mathrm{Mod}(\mathcal{C}^{op}):=\mathrm{Fun}_{k}(\mathcal{C}^{op},\mathrm{Mod}(k))$ whose objects will sometimes be denoted as $M_{\mathcal{C}}$. It is well-known that $\mathrm{Mod}(\mathcal{C})$ is an abelian category with enough projectives and injectives; see for example,\cite[Proposition 2.3]{MitBook} on page 99 and also page 102 in \cite{MitBook}.\\
 
If $\mathcal{C}$ and $\mathcal{D}$ are $k$-categories, B. Mitchell defined in \cite{Mitchelring} the $k$-category tensor product  $\mathcal{C}\otimes_{k}\mathcal{D}$ with objects that are those of $\mathcal{C}\times \mathcal{D}$, and the set of morphisms from $(C,D)$ to $(C',D')$ is the tensor product of $k$-vector spaces $\mathcal{C}(C,C')\otimes_{k}\mathcal{D}(D,D')$. The $k$-bilinear composition in $\mathcal{C}\otimes_{k} \mathcal{D}$ is given as follows: $(f_{2}\otimes g_{2})\circ (f_{1}\otimes g_{1}):=(f_{2}\circ f_{1})\otimes(g_{2}\circ g_{1})$ 
for all $f_{1}\otimes g_1\in \mathcal{C}(C,C')\otimes_k \mathcal{D}(D,D')$ and  $f_{2}\otimes g_2\in\mathcal{C}(C',C'')\otimes_{k} \mathcal{D}(D',D'')$. We will say that a module $M\in \mathrm{Mod}(\mathcal{D}^{op}\otimes_k\mathcal{C})$ is a $\mathcal{C}-\mathcal{D}$-bimodule. In this case we denote $_\mathcal{C}M_{\mathcal{D}}$ to enfatize the action of $\mathcal{C}$ and $\mathcal{D}$ in $M$. \\

Now we recall an important construction given in  \cite{Mitchelring} on p. 26 that will be used throughout this paper. Let $\mathcal{C}$ and $\mathcal{A}$ be  $k$-categories where $\mathcal{A}$ is cocomplete. The evaluation $k$-functor $E:\mathrm{Fun}_{k}(\mathcal{C}^{op},\mathcal{A})\otimes_{k}\mathcal{C}\longrightarrow \mathcal{A}$ can be extended to a $k$-functor
\begin{equation}\label{tensorMitchel}
-\otimes_{\mathcal{C}}-:\mathrm{Fun}_{k}(\mathcal{C}^{op},\mathcal{A})\otimes_{k}\mathrm{Mod}(\mathcal{C})\longrightarrow \mathcal{A}.
\end{equation}
By definition, we have an isomorphism $F\otimes_{\mathcal{C}}\mathcal{C}(X,-)\simeq F(X)$ for all $X\in \mathcal{C}$,
which is natural in $F$ and $X$.\\

By taking $\mathcal{A}=\mathrm{Mod}(k)$ in equation \ref{tensorMitchel},  we have a functor 
$$-\otimes_{\mathcal{C}}-:\mathrm{Mod}(\mathcal{C}^{op})\times \mathrm{Mod}(\mathcal{C})\longrightarrow \mathrm{Mod}(k).$$
For properties of this tensor product we refer the reader to \cite{AuslanderRep1}.
We can extend this functor by considering bimodules in the following way. For three $k$-categories $\mathcal{B}$, $\mathcal{C}$ and $\mathcal{D}$, we have the functor
$$-\boxtimes_{\mathcal{C}}-:\mathrm{Mod}(\mathcal{C}^{op}\otimes_k\mathcal{B})\times \mathrm{Mod}(\mathcal{D}^{op}\otimes_k\mathcal{C})\longrightarrow \mathrm{Mod}(\mathcal{D}^{op}\otimes_k\mathcal{B}).$$
given by $_{\mathcal{B}}M\boxtimes_{\mathcal{C}}N_{\mathcal{D}}(D,B)=M(-,B)\otimes_{\mathcal{C}}N(D,-)$ for $(D,B)\in \mathcal{D}^{op}\otimes_k\mathcal{B}$. When $\mathcal{B}$ or $\mathcal{D}$ are $\mathcal{C}_k$, where $\mathcal{C}_k$ is the one point category with endomorphism set $k$, we consider the isomorphisms $\mathcal{C}\otimes_k\mathcal{C}_k\cong\mathcal{C}\cong \mathcal{C}_k\otimes_k\mathcal{C}$. Therefore we can express the functor $-\boxtimes_\mathcal{C}-$ by considering bimodules in just one side, or the functor $-\otimes_{\mathcal{C}}-$ as a particular case of the functor $-\boxtimes_{\mathcal{C}}-$. Another functors we will use are the functor
\[\mathbb{HOM}_{\mathrm{\mathrm{C}}}(-,-):\mathrm{Mod}(\mathcal{B}^{op}\otimes_{k}\mathcal{C})\times \mathrm{Mod}(\mathcal{D}^{op}\otimes_k\mathcal{C})\to \mathrm{Mod}(\mathcal{D}^{op}\otimes_k\mathcal{B})\]
given by $\mathbb{HOM}(M_{\mathcal{B}},N_{\mathcal{D}})(D,B)=\mathrm{Mod}(\mathcal{D}^{op}\otimes_k\mathcal{B})(M(B,-),N(D,-))$ and the functor
\[\mathbb{HOM}_{\mathrm{\mathrm{C}^{op}}}(-,-):\mathrm{Mod}(\mathcal{C}^{op}\otimes_{k}\mathcal{D})\times \mathrm{Mod}(\mathcal{C}^{op}\otimes_k\mathcal{B})\to \mathrm{Mod}(\mathcal{D}^{op}\otimes_k\mathcal{B})\] given by $\mathbb{HOM}_{\mathcal{C}^{op}}(_{\mathcal{D}}M,_{\mathcal{B}}N)(D,B)=\mathrm{Mod}(\mathcal{D}^{op}\otimes_k\mathcal{B})(M(-,D),N(-,B))$. Analogously as in the tensor product, we can consider bifunctors in just one side in a natural form. We will usually denote by $\mathcal{C}$ to the $\mathcal{C}^{op}\otimes_k\mathcal{C}$-module $\mathcal{C}(-,-)$. If there are $k$-functors $\phi:\mathcal{B}\to \mathcal{C}$ and $\psi:\mathcal{D}\to \mathcal{C}$, we can denote by $_\mathcal{B}\mathcal{C}_{\mathcal{D}}$ to the $\mathcal{B}-\mathcal{D}$-bimodule $\mathcal{C}(-,-)\circ (\psi^{op}\times \phi)$. When the functors are obvious like inclusions or projections we will not especify the functors.

\subsection{Derived categories}
Let $\mathcal{A}$ be an additive category, and let $K(\mathcal{A})$ be the homotopy category of $\mathcal{A}$. The subcategories
$K^{+}(\mathcal{A})$, $K^{-}(\mathcal{A})$ and $K^{b}(\mathcal{A})$ of $K(\mathcal{A})$ are
generated by the bounded below complexes, the bounded above complexes, and the bounded complexes, respectively.
For an abelian category $\mathcal{A}$, the derived category $\mathbf{D}(\mathcal{A})$ (resp. $\mathbf{D}^{+}(\mathcal{A})$, $\mathbf{D}^{-}(\mathcal{A})$ and $\mathbf{D}^{b}(\mathcal{A})$) is the quotient of $K(\mathcal{A})$ (resp. $K^{+}(\mathcal{A})$, $K^{-}(\mathcal{A})$ and $K^{b}(\mathcal{A})$) by the multiplicative set of quasi-isomorphisms.
Therefore $K^{\ast}(\mathcal{A})$ and $\mathbf{D}^{\ast}(\mathcal{A})$ are triangulated categories, where 
$$\ast=\text{nothing},+,-,\,\text{or}\,b,$$
see (\cite{Hartshorne}, \cite{Verdier}).\\
In general, we denote $K^{\ast}(\mathcal{A})$ as a localizing subcategory of $K(\mathcal{A})$, meaning that $K^{\ast}(\mathcal{A})$  is a full triangulated subcategory of $K(\mathcal{A})$, and the functor $\mathbf{D}^{\ast}(\mathcal{A})\longrightarrow  \mathbf{D}(\mathcal{A})$ is fully faithfull, where $\mathbf{D}^{\ast}(\mathcal{A})$ is the quotient of $K^{\ast}(\mathcal{A})$ by a multiplicative set of quasi-isomorphisms ( \cite[I, Sect. 5]{Hartshorne}, \cite[II, Sect. 1, No. 1]{Verdier}).
For further details on the triangulated structure of  $\mathbf{D}^{\ast}(\mathcal{A})$ see, for example, \cite{Gelfand}.
Let $\mathcal{T}$ be a triangulated category with equivalence  $\Sigma$. A non-empty full subcategory $\mathcal{S}$ of $\mathcal{T}$ is a triangulated subcategory if the following conditions hold.
\begin{enumerate}
\item [(a)] $\Sigma^{n}(X)\in \mathcal{S}$ for all $X\in \mathcal{S}$ and for all $n\in \mathbb{Z}$,

\item [(b)] Let $\xymatrix{X\ar[r] & Y\ar[r] & Z\ar[r] & \Sigma(X)}$ be a triangle in $\mathcal{T}$.  If two objects of
$\{X, Y, Z\}$ belong to $\mathcal{S}$, then also the third.
\end{enumerate}
A triangulated subcategory $\mathcal{S}$ of $\mathcal{T}$ is $\textbf{thick}$ if, for any morphisms  $\xymatrix{X\ar[r]^{\pi} & Y\ar[r]^{i} & X}$  in $\mathcal{T}$ where  $\pi\circ i=1_{Y}$ and $X\in \mathcal{S}$, it follows that $Y\in \mathcal{S}$

A  $\textbf{two}$ $\textbf{sided}$ $\textbf{ideal}$  $\mathcal{I}(-,-)$ of $\mathcal{C}$ is a $k$-subfunctor of the two variable functor $\mathcal{C}(-,-):\mathcal{C}^{op}\otimes_{k}\mathcal{C}\rightarrow\mathrm{Mod}(k)$, such that the following conditions hold: (a) if $f\in \mathcal{I}(X,Y)$ and $g\in\mathcal C(Y,Z)$, then  $gf\in \mathcal{I}(X,Z)$; and (b)
if $f\in \mathcal{I}(X,Y)$ and $h\in\mathcal{C}(U,X)$, then  $fh\in \mathcal{I}(U,Z)$. If $\mathcal{I}$ is a two-sided ideal,  we can form the $\textbf{quotient category}$  $\mathcal{C}/\mathcal{I}$, whose objects are those of $\mathcal{C}$ and where $(\mathcal{C}/\mathcal{I})(X,Y):=\mathcal{C}(X,Y)/\mathcal{I}(X,Y)$, with composition induced by that of $\mathcal{C}$ (see \cite{Mitchelring}). There is a canonical projection functor $\pi:\mathcal{C}\rightarrow \mathcal{C}/\mathcal{I}$ such that $\pi(X)=X$ for all $X\in \mathcal{C}$ and  $\pi(f)=f+\mathcal{I}(X,Y):=\bar{f}$ for all $f\in \mathcal{C}(X,Y)$. We also recall that there exists a canonical isomorphism of categories $(\mathcal{C}/\mathcal{I})^{op}\simeq \mathcal{C}^{op}/\mathcal{I}^{op}$.\\

For $M\in\mathrm{Mod}(\mathcal{C}^{op})$, we denote by $\mathrm{Tor}_{i}^{\mathcal{C}}(M,-):\mathrm{Mod}(\mathcal{C})\longrightarrow \mathrm{Mod}(k)$ the $i$-th left derived functor of $M\otimes_{\mathcal{C}} -$. For $N\in \mathrm{Mod}(\mathcal{C})$ we now denote by $\mathrm{Ext}^{i}_{\mathcal{C}}(N,-):\mathrm{Mod}(\mathcal{C})\longrightarrow \mathrm{Mod}(k)$ the $i$-th right derived functor of $\mathrm{Mod}(\mathcal{C})(N,-):\mathrm{Mod}(\mathcal{C})\longrightarrow \mathrm{Mod}(k)$.\\

\begin{defi}
For $M\in\mathrm{Mod}(\mathcal{C}
^{op}\otimes_k\mathcal{B})$, we denote by  $M\boxtimes^L_{\mathcal{C}}-:D^-(\mathrm{Mod}(\mathcal{C}))\to D^-(\mathrm{Mod}(\mathcal{B}))$ the left derived functor of $M\boxtimes_{\mathcal{C}}-:\mathrm{Mod}(\mathcal{C})\to \mathrm{Mod}(\mathcal{B})$ and by  $\mathbb{TOR}_{i}^{\mathcal{C}}(M,-):\mathrm{Mod}(\mathcal{C})\rightarrow \mathrm{Mod}(\mathcal{B})$ its $i$-th left derived functor. For $N\in\mathrm{Mod}(\mathcal{D}^{op}\otimes_k\mathcal{C})$, we denote by $\mathbb{RHOM}_{\mathcal{C}}(N,-):D^+(\mathrm{Mod}(\mathcal{C}))\to D^+(\mathrm{Mod}(\mathcal{D}))$ the right derived functor of $\mathbb{HOM}_\mathcal{C}(N,-):\mathrm{Mod}(\mathcal{C})\to \mathrm{Mod}(\mathcal{D})$ and by        $\mathbb{EXT}^{i}_{\mathcal{C}}(N,-):\mathrm{Mod}(\mathcal{C})\rightarrow \mathrm{Mod}(\mathcal{D})$ its $i$-th right derived functor. 
\end{defi}

We have the following description of the i-th derived functors defined above
\begin{obs}\label{descEXT}
Consider the functors given above $\mathbb{TOR}_{i}^{\mathcal{C}}(M,-):\mathrm{Mod}(\mathcal{C})\longrightarrow \mathrm{Mod}(\mathcal{B})$ and $\mathbb{EXT}^{i}_{\mathcal{C}}(N,-):\mathrm{Mod}(\mathcal{C})\longrightarrow \mathrm{Mod}(\mathcal{D})$. The following  holds for $K\in\mathrm{Mod}(\mathcal{C})$.
\begin{enumerate}
\item [(a)] For $B\in \mathcal{B}$ we have that $\mathbb{TOR}_{i}^{\mathcal{C}}(M,K)(B)=\mathrm{Tor}_{i}^{\mathcal{C}}(M(-,B),K)$. 
\item [(b)] For $D\in\mathcal{D}$ we have that $\mathbb{EXT}^{i}_{\mathcal{C}}(N,K)(D)=\mathrm{Ext}^{i}_{\mathcal{C}}(N(D,-),K)$.

\end{enumerate}
\end{obs}

\begin{proof}
    \begin{enumerate}
    \item[(a)] Let $B\in \mathcal{B}$ and $\cdots\to P^{-1}\xrightarrow{d^{-1}} P^{0}\to K\to 0$ be a projective resolution of $K$. Then 
    \begin{align*}
        \mathbb{TOR}_i^{\mathcal{C}}(M,K)(B)=&\mathrm{H}^{-i}(M\boxtimes_{\mathcal{C}}P^{\bullet})(B)\\
        =&\frac{\mathrm{Ker}(M\boxtimes_{\mathcal{C}}d^{-i})(B)}{\mathrm{Im}(M\boxtimes_{\mathcal{C}}d^{-i-1})(B)}\\
        =&\frac{\mathrm{Ker}(M(-,B)\otimes_{\mathcal{C}}d^{-i})}{\mathrm{Im}((M(-,B)\otimes_{\mathcal{C}}d^{-i-1})}\\
        =&\mathrm{Tor}^{\mathcal{C}}_i(M(-,B)\otimes_{\mathcal{C}}K)
    \end{align*}
    \item[(b)] It is analogously as (a).
    \end{enumerate}
\end{proof}

\subsection{Restricting functors to the bounded derived category}

We now give the following definition, see for example first paragraph in p. 85 in \cite{Borel}
\begin{defi}
Let $F:\mathcal{A} \longrightarrow \mathcal{B}$ be a functor between abelian categories. We say that $F$ has \textbf{finite left cohomological dimension} if there exists an integer $n\geq 0$ such that 
$$L_{i}F(A)=H^{-i}(L^{-}F(A))=0$$
for all $A\in \mathcal{A}$ and for all $i>n$ ( we consider $A$ as a complex concentrated in zero degree), where
$L^{-}F:\mathbf{D}^{-}(\mathcal{A})\longrightarrow  \mathbf{D}^{-}(\mathcal{B})$ is the left derived functor of $F$. Dually, we say that $F$ has \textbf{finite right cohomological dimension} if there exists an integer $n\geq 0$ such that 
$$R_{i}F(A)=H^{i}(R^{+}F(A))=0$$
for all $A\in \mathcal{A}$ and for all $i>n$, where
$R^{+}F:\mathbf{D}^{+}(\mathcal{A})\longrightarrow  \mathbf{D}^{+}(\mathcal{B})$ is the right derived functor of $F$.
\end{defi}

The importance of finite left (co)homological dimension is that it allow us to restrict derived functors to the bounded derived categories.

\begin{lema}\label{lema1Cline-Parshall}
Let $\mathcal{A}$ and $\mathcal{B}$ be abelian categories such that $\mathcal{A}$ has enough injectives and $\mathcal{B}$ has enough projectives. Let $F:\mathcal{A}\longrightarrow \mathcal{B}$ and $G:\mathcal{B}\longrightarrow \mathcal{A}$
additive functors such that
\begin{enumerate}
\item [(a)] $F$ is right adjoint to $G$,

\item [(b)] $F$ has finite right cohomological dimension and $G$ has finite left cohomological dimension.
\end{enumerate}
Then $R^{+}F:\mathbf{D}^{b}(\mathcal{A})\longrightarrow \mathbf{D}^{b}(\mathcal{B})$ is right adjoint to  $L^{-}G: \mathbf{D}^{b}(\mathcal{B})\longrightarrow \mathbf{D}^{b}(\mathcal{A})$.
\end{lema}
\begin{proof}
See \cite[Lemma 1.1]{CPS1} in p. 399.
\end{proof}


\section{Recollements of Triangulated Categories.}

Recollements were introduced by A. A. Beilinson, J. Bernstein and P.Deligne in \cite{Recoll}. Here we give the definition of a recollement of triangulated categories and some results we need.

\begin{defi}
    Let $\mathcal{T}_1, \mathcal{T}, \mathcal{T}_2$ be triangulated categories. A recollement of $\mathcal{T}$ relative to $\mathcal{T}_1$ and $\mathcal{T}_2$ is given by 
    \[\begin{tikzcd}
	{\mathcal{T}_1} && {\mathcal{T}} && {\mathcal{T}_2}
	\arrow["{i_{\ast}=i_{!}}", from=1-1, to=1-3]
	\arrow["{i^{!}}", shift left=3, curve={height=-12pt}, from=1-3, to=1-1]
	\arrow["{i^{\ast}}"', shift right=3, curve={height=12pt}, from=1-3, to=1-1]
	\arrow["{j^{\ast}=j^{!}}", from=1-3, to=1-5]
	\arrow["{j_{!}}"', shift right=3, curve={height=12pt}, from=1-5, to=1-3]
	\arrow["{j_{\ast}}", shift left=3, curve={height=-12pt}, from=1-5, to=1-3]
\end{tikzcd}\]
such that the following statements holds
\begin{enumerate}
    \item[R1)] $(i^{\ast},i_{\ast})$, $(i_{!},i^{!})$, $(j_!,j^!)$ and $(j^{\ast},j_{\ast})$ are adjoint pairs of triangulated functors;
    \item[R2)] $i_{\ast}$, $j_!$ and $j_{\ast}$ are full embeddings;
    \item[R3)] $j^{\ast}i_{\ast}=0$;
    \item[R4)] For each $X\in\mathcal{T}$, there are distinguished triangles
    \[i_!i^!X\to X\to j_{\ast}j^{\ast}X\to i_!i^!X[1]\]
    and
    \[i_{\ast}i^{\ast}X\to X\to j_{!}j^{!}X\to i_{\ast}i^{\ast}X[1]\]
    where the arrows to and from $X$ are the counit and unit respectively.
\end{enumerate}
\end{defi}

The following two theorems tell us that if we have one half of a possible recollement, then we can complete to a recollement.

\begin{teo}[\cite{parshall1}, theorem 1.1]\label{recoll1}
    Let $j^{\ast}:\mathcal{D}\to \mathcal{D}''$ be a morphism of triangulated categories. Assume that $j^{\ast}$ has both a right adjoint $j_{\ast}$ and a left adjoint $j_{!}$, and both $j_{\ast}, j_!$ are full embeddings.
    \begin{enumerate}
        \item[(i)] Let $\mathcal{D}'=\mathrm{Ker}(j^{\ast})$ be the full, triangulated subcategory of $\mathcal{D}$ whose objects consists of those objects $X$ satisfying $j^{\ast}X\cong 0$, and let $i_{\ast}:\mathcal{D}'\to \mathcal{D}$ be the inclusion functor. Then $i_{\ast}$ has both a left adjoint $i^{\ast}$ and a right adjoint $i^!$. 
        \item[(ii)] For convenience, put $i_!=i_{\ast}$ and $j^{!}=j^{\ast}$. Then for each object $X$ in $\mathcal{D}$ there are distinguished triangles
        \begin{enumerate}
            \item[(a)] $i_!i^!X\to X\to j_{\ast}j^{\ast}X\to $,
            \item[(b)] $j_!j^!X\to X\to i_{\ast}i^{\ast}X\to$.
        \end{enumerate}
        \item[(iii)] Any distinguished triangle $X'\to X\to X''\to $ in $\mathcal{D}$ with $X'\in \mathrm{Im}(i_!)=\mathcal{D}'$ (resp. $\mathrm{Im}(j_!)$) and $X''\in\mathrm{Im}(j_{\ast})$ (resp. $\mathrm{Im}(i_{\ast})$) is isomorphic to the distinguished triangle given in (ii).(a) (resp. (ii).(b)) above.
    \end{enumerate}
\end{teo}

\begin{teo}[\cite{parshall2}, theorem 2.1]\label{recoll2}
    Let $i_{\ast}:\mathcal{D}'\to\mathcal{D}$ be a morphism of triangulated categories which is a full embedding. Assume that $i_{\ast}$ has both a left adjoint $i^{\ast}$ and a right adjoint $i^{!}$.
    \begin{enumerate}
        \item[(i)] Let $\mathcal{E}$ be the strict image of $\mathcal{D}'$ under $i_{\ast}$. Then, $\mathcal{E}$ is a thick subcategory of $\mathcal{D}'$, and the quotient morphism $j^{\ast}:\mathcal{D}'\to \mathcal{D}'/\mathcal{E}$ has both a left adjoint $j_{!}$ and a right adjoint $j_{\ast}$, each of which is a full embedding.

        \item[(ii)] For convenience, put $i_!=i_{\ast}$ and $j^{!}=j^{\ast}$. Then for each object $X$ in $\mathcal{D}$ there are distinguished triangles
        \begin{enumerate}
            \item[(a)] $i_!i^!X\to X\to j_{\ast}j^{\ast}X\to $,
            \item[(b)] $j_!j^!X\to X\to i_{\ast}i^{\ast}X\to$.
        \end{enumerate}
        \item[(iii)] Any distinguished triangle $X'\to X\to X''\to $ in $\mathcal{D}$ with $X'\in \mathrm{Im}(i_!)=\mathcal{E}$ (resp. $\mathrm{Im}(j_!)$) and $X''\in\mathrm{Im}(j_{\ast})$ (resp. $\mathrm{Im}(i_{\ast})$) is isomorphic to the distinguished triangle given in (ii).(a) (resp. (ii).(b)) above.
    \end{enumerate}
\end{teo}

The following result taken from \cite{liu} tells us how to induce a recollement of verdier quotients.

\begin{prop}[\cite{liu}, Proposition 2.5]\label{quotientrecoll}
Let
\[\begin{tikzcd}
	{\mathcal{D}_1} && {\mathcal{D}} && {\mathcal{D}_2}
	\arrow["{i_{\ast}=i_{!}}", from=1-1, to=1-3]
	\arrow["{i^{!}}", shift left=3, curve={height=-12pt}, from=1-3, to=1-1]
	\arrow["{i^{\ast}}"', shift right=3, curve={height=12pt}, from=1-3, to=1-1]
	\arrow["{j^{\ast}=j^{!}}", from=1-3, to=1-5]
	\arrow["{j_{!}}"', shift right=3, curve={height=12pt}, from=1-5, to=1-3]
	\arrow["{j_{\ast}}", shift left=3, curve={height=-12pt}, from=1-5, to=1-3]
\end{tikzcd}\]
be a recollement of triangulated categories and $\mathcal{T}$ be a thick subcategory of $\mathcal{D}$. If $i_{\ast}i^{\ast}(\mathcal{T})\subseteq \mathcal{T}$ and $j_{\ast}j^{\ast}(\mathcal{T})\subseteq \mathcal{T}$, then there exists a recollement of triangulated categories
\[\begin{tikzcd}
	{\mathcal{D}_1/i^{\ast}(\mathcal{T})} && {\mathcal{D}/\mathcal{T}} && {\mathcal{D}_2/j^{\ast}(\mathcal{T})}
	\arrow["{\tilde{i_{\ast}}}", from=1-1, to=1-3]
	\arrow["{\tilde{i^{!}}}", shift left=3, curve={height=-12pt}, from=1-3, to=1-1]
	\arrow["{\tilde{i^{\ast}}}"', shift right=3, curve={height=12pt}, from=1-3, to=1-1]
	\arrow["{\tilde{j^{\ast}}}", from=1-3, to=1-5]
	\arrow["{\tilde{j_{!}}}"', shift right=3, curve={height=12pt}, from=1-5, to=1-3]
	\arrow["{\tilde{j_{\ast}}}", shift left=3, curve={height=-12pt}, from=1-5, to=1-3]
\end{tikzcd}\]

\end{prop}

\subsection{Recollements for Module Categories.}

\begin{lema}\label{Gedrich}
    Let $\mathcal{A}$ and $\mathcal{B}$ be additive categories with $\mathcal{A}$ an $AB3$ category and $\{F_n:\mathcal{A}\to \mathcal{B}\}_{n\in \mathbb{N}}$ a family of additive functors such that for each $A\in \mathcal{A}$, $F_n(A)=0$ for $n>>0$. Then $F_n=0$ for $n>>0$.
\end{lema}

\begin{proof}
   We will prove the contrapositive. Suppose that for each $n\in\mathbb{N}$ there exists $s_n> n$ such that $F_{s_n}\neq 0$. Then, for each $n$ there exists $A_n\in \mathcal{A}$ such that $F_{s_n}(A_n)\neq 0$. Let $A:=\amalg_{m\in\mathbb{N}}A_m$. For each $n$, $F_{s_n}(A)=F_{s_n}(A_n)\amalg F_{s_n}(\amalg_{m\neq n}A_m)\neq 0$ since $F_{s_n}(A_n)\neq 0$.
\end{proof}

The following results are generalizations of those from \cite{parshall2}.

For the following we will consider a $k$-category $\mathcal{C}$, $\mathcal{I}\trianglelefteq \mathcal{C}$ an ideal, $\pi:\mathcal{C}\to \mathcal{C}/\mathcal{I}$ the projection. Since the induced functor $\pi_{\ast}:\mathrm{Mod}(\mathcal{C}/\mathcal{I})\to \mathrm{Mod}(\mathcal{C})$ is exact, then it induces an exact functor $i_{\ast}:D^b(\mathrm{Mod}(\mathcal{C}/\mathcal{I}))\to D^b(\mathrm{Mod}(\mathcal{C}))$. Recall that the right derived functor of $\mathbb{HOM}_{\mathcal{C}}(\mathcal{C}/\mathcal{I},-): \mathrm{Mod}(\mathcal{C})\to \mathrm{Mod}(\mathcal{C}/\mathcal{I})$, denoted by $\mathbb{RHOM}_\mathcal{C}(\mathcal{C}/\mathcal{I},-):D^+(\mathrm{Mod}(\mathcal{C}))\to D^+(\mathrm{Mod}(\mathcal{C}/\mathcal{I}))$, is calculated by  $\mathbb{RHOM}_{C}(\mathcal{C}/\mathcal{I},X^{\bullet}):=\mathbb{HOM}_{\mathcal{C}}(\mathcal{C}/\mathcal{I},I^{\bullet})$ where $I^{\bullet}$ is an injective coresolution of the complex $X^{\bullet}$. Also we have the left derived functor of $\mathcal{C}/\mathcal{I}\boxtimes_{\mathcal{C}}-:\mathrm{Mod}(\mathcal{C})\to \mathrm{Mod}(\mathcal{C}/\mathcal{I})$, denoted by $\mathcal{C}/\mathcal{I}\boxtimes^L_{\mathcal{C}}-:D^-(\mathrm{Mod}(\mathcal{C}))\to D^-(\mathrm{Mod}(\mathcal{C}/\mathcal{I}))$, where $\mathcal{C}/\mathcal{I}\boxtimes^L_{\mathcal{C}}X^{\bullet}:=\mathcal{C}/\mathcal{I}\boxtimes_{\mathcal{C}}P^{\bullet}$ where $P^{\bullet}$ is a projective resolution of the complex $X^{\bullet}$.

\begin{defi}
    Let $\mathcal{C}$, $\mathcal{D}$ be $k$-categories and $M\in\mathrm{Mod}(\mathcal{D}^{op}\otimes_k\mathcal{C})$.
    \begin{enumerate}
        \item We say that $pd(_{\mathcal{C}}M)\leq n$ if $\mathbb{EXT}^k_{\mathcal{C}}(M,-)=0$ for $k> n$.
        \item We say that $pd(_{\mathcal{C}}M)<\infty$ if there exists $n\in\mathbb{N}$ such that $pd_{\mathcal{C}}M\leq n$, otherwise we write $pd(_{\mathcal{C}}M)=\infty$.
        \item We say that $pdM_{\mathcal{D}}\leq n$ if $\mathbb{EXT}^k_{\mathcal{D}^{op}}(M,-)=0$ for $k>n$.
        \item We say that $pd(M_{\mathcal{D}})< \infty$ if there exists $n\in\mathbb{N}$ such that $pd(M_{\mathcal{D}})\leq n$, otherwise we say that $pd(M_{\mathcal{D}})=\infty$.
        \item $gl.dim(\mathscr{C}):=sup\{pd(_{\mathcal{C}}F):F\in\mathrm{Mod}(\mathscr{C})\}$.
    \end{enumerate}   
\end{defi}

Observe that analogously as in modules, if $pd(M_{\mathcal{D}})<\infty$, then $\mathbb{TOR}_{\mathcal{D}}^k(M,-)$ for $k>>0.$

The following is an adjustment of the theorem 3.1 of \cite{parshall2} in this context

\begin{teo}\label{Recollalg}
    Let $\mathcal{C}$, $\mathcal{I}$ and $i_{\ast}:D^b(\mathrm{Mod}(\mathcal{C}/\mathcal{I}))\to D^b(\mathrm{Mod}(\mathcal{C}))$ be as above.
    \begin{enumerate}
        \item[(1)] The morphism of triangulated categories $i_!=i_{\ast}$ has a right adjoint $i^!$ satisfying $i^!i_!\cong 1_{D^b(\mathrm{Mod}(\mathcal{C}/\mathcal{I}))}$ via the unit if and only if
        \begin{enumerate}
            \item[(a)] $\mathbb{EXT}^n_{\mathcal{C}}(\mathcal{C}/\mathcal{I},F\circ \pi)=0$ for all $n>0$ and all free module $F\in\mathrm{Mod}(\mathcal{C}/\mathcal{I})$;
            \item[(b)] $pd(_{\mathcal{C}}\mathcal{C}/\mathcal{I})<\infty$.
        \end{enumerate}
        When these hold, the right adjoint is defined by $i^!=\mathbb{RHOM}_{\mathcal{C}}(\mathcal{C}/\mathcal{I},-)$ and $i_{\ast}=i_!$ is a full embedding.
        \item[(2)] $i_{\ast}$ satisfies the recollement setup if and only if conditions $(a)$ and $(b)$ of $(1)$ above hold and $\mathbb{TOR}^{\mathcal{C}}_n(\mathcal{C}/\mathcal{I},-)=0$ for $n>>0$. In this case the left adjoint to $i_{\ast}$ is defined by $i^{\ast}=\mathcal{C}/\mathcal{I}\boxtimes^{L}_{\mathcal{C}}-$.
    \end{enumerate}
\end{teo}

\begin{proof}
\begin{enumerate}
    \item[(1)] $\Rightarrow)$ Suppose that $i_{\ast}$ has a right adjoint $i^!$ satisfying $i^!i_!\cong 1_{D^b(\mathcal{B})}$. Then $i_{\ast}=-\circ \pi$ is a full embedding, so that, for $C\in\mathcal{C}$ 
    \begin{align*}
        \mathbb{EXT}_{\mathcal{C}}^n(\mathcal{C}/\mathcal{I},F\circ \pi)(C)=&\mathrm{Ext}^n_{\mathcal{C}}(\mathcal{C}/\mathcal{I}(C,-)\circ \pi,F\circ \pi)\\
        =&D^b(\mathrm{Mod}(\mathcal{C}))(\mathcal{C}/\mathcal{I}(C,-)\circ \pi,F\circ \pi[n])\\
        \cong &D^b(\mathrm{Mod}(\mathcal{C}/\mathcal{I}))(\mathcal{C}/\mathcal{I}(C,-),F[n])\\
        =&\mathbb{EXT}_{\mathcal{C}/\mathcal{I}}^n(\mathcal{C}/\mathcal{I},F)=0
    \end{align*}
    for all $n>0$ and all free modules $F\in\mathrm{Mod}(\mathcal{C}/\mathcal{I})$. Also, if $M\in \mathrm{Mod}(\mathcal{C})$, then 
    \begin{align*}
        \mathbb{EXT}^n_{\mathcal{C}}(\mathcal{C}/\mathcal{I},M)(C)=&D^b(\mathrm{Mod}(\mathcal{C}))(i_{!}(\mathcal{C}/\mathcal{I}(C,-)),M[n])\\
        \cong & D^b(\mathrm{Mod}(\mathcal{C}/\mathcal{I}))(\mathcal{C}/\mathcal{I}(C,-),i^!M[n])\\
        = &0
    \end{align*}
    for $n>>0$, since $i^!M[n]\cong (i^!M)[n]$ and $i^!M\in D^b(\mathrm{Mod}(\mathcal{C}/\mathcal{I}))$ by hyphotesis. By lemma \ref{Gedrich}, we have that $\mathbb{EXT}_{\mathcal{C}}^n(\mathcal{C}/\mathcal{I},-)=0$ for $n>>0$. Whence $pd_{\mathcal{C}}(\mathcal{C}/\mathcal{I})<\infty$.\\
    $\Leftarrow)$ Conversely, suppose that $\mathbb{EXT}^n_{\mathcal{C}}(\mathcal{C}/\mathcal{I},F\circ \pi)=0$ for $n>0$ and all free modules $F\in \mathrm{Mod}(\mathcal{C}/\mathcal{I})$ and that $pd_{\mathcal{C}}(\mathcal{C}/\mathcal{I})<\infty$. Let us see that $\mathbb{RHOM}_{\mathcal{C}}(\mathcal{C}/\mathcal{I},X)\in D^b(\mathrm{Mod}(\mathcal{C}/\mathcal{I}))$ if $X\in D^b(\mathrm{Mod}(\mathcal{C}))$. Let $X\in D^b(\mathrm{Mod}(\mathcal{C}))$ and $I^{\bullet}\in D^+(\mathrm{Mod}(\mathcal{C}))$ an injective coresolution of $X$. Then, $\mathbb{RHOM}_{\mathcal{C}}(\mathcal{C}/\mathcal{I},X)=\mathbb{HOM}_{\mathcal{C}}(\mathcal{C}/\mathcal{I},I^{\bullet})$. Since $pd_{\mathcal{C}}(\mathcal{C}/\mathcal{I})=n<\infty$, then $0=\mathbb{EXT}^m_{\mathcal{C}}(\mathcal{C}/\mathcal{I},X)=H^m(\mathbb{HOM}_{\mathcal{C}}(\mathcal{C}/\mathcal{I},I^{\bullet}))$ for all $m>n$ and then $\mathbb{RHOM}_{\mathcal{C}}(\mathcal{C}/\mathcal{I},X)\in D^b(\mathrm{Mod}(\mathcal{C}/\mathcal{I})))$.\\
    Let us see that $i^!=\mathbb{RHOM}_{\mathcal{C}}(\mathcal{C}/\mathcal{I},-)$ is right adjoint to $i_!$. For this we use that $\mathbb{HOM}_{\mathcal{C}}(\mathcal{C}/\mathcal{I},-)$ is right adjoint to $_{\mathcal{C}}\mathcal{C}/\mathcal{I}\boxtimes_{\mathcal{C}/\mathcal{I}}-$ at the level of homotopy categories and $\mathbb{HOM}_{\mathcal{C}}(\mathcal{C}/\mathcal{I},I)$ is $\mathcal{C}/\mathcal{I}$-inyective if $I$ is $\mathcal{C}$-injective (This follows from the fact that $\mathbb{HOM}_{\mathcal{C}}(\mathcal{C}/\mathcal{I},-)$ is right adjoint of an exact functor and then preserves injectives). Let $X\in D^b(\mathrm{Mod}(\mathcal{C}))$ and $I^{\bullet}$ be an injective coresolution of $X$, then we have $i^!X\cong \mathbb{HOM}_{\mathcal{C}}(\mathcal{C}/\mathcal{I},I^{\bullet})$. Thus, for $Y\in D^b(\mathrm{Mod}(\mathcal{C}/\mathcal{I}))$, we have
    \begin{align*}
        D^b(\mathrm{Mod}(\mathcal{C}/\mathcal{I}))(Y,i^!X)\cong & K^+(\mathrm{Mod}(\mathcal{C}/\mathcal{I}))(Y,\mathbb{HOM}_{\mathcal{C}}(\mathcal{C}/\mathcal{I},I^{\bullet}))\\
        \cong & K^+(\mathrm{Mod}(\mathcal{C}))(_{\mathcal{C}}\mathcal{C}/\mathcal{I}\boxtimes_{\mathcal{C}/\mathcal{I}}Y,I^{\bullet})\\
        = & K^+(\mathrm{Mod}(\mathcal{C}))(Y\circ \pi,I^{\bullet})\\
        \cong &D^b(\mathrm{Mod}(\mathcal{C}))(i_{!}Y,X).
    \end{align*}
Next, to verify that $i^!i_!\cong 1_{D^b(\mathrm{Mod}(\mathcal{C}/\mathcal{I}))}$ via the unit. Let $M\in \mathrm{Mod}(\mathcal{C}/\mathcal{I})$ and $0\to Q\to F\to M\to 0$ a short exact sequence in $\mathrm{Mod}(\mathcal{C}/\mathcal{I})$ with $F$ free. It follows from the long exact sequence of $\mathbb{HOM}$ that $\mathbb{EXT}^n_{\mathcal{C}}(\mathcal{C}/\mathcal{I},M\circ \pi)\cong \mathbb{EXT}^{n+1}_{\mathcal{C}}(\mathcal{C}/\mathcal{I},Q\circ \pi)$ for $n>0$ since $\mathbb{EXT}^q_{\mathcal{C}}(\mathcal{C}/\mathcal{I},F\circ\pi)=0$ for $q>0$ by hyphotesis. Thus, applying a simetric argument with $Q$ and the fact that $pd_{\mathcal{C}}\mathcal{C}/\mathcal{I}<\infty$, we have that $\mathbb{EXT}^n_{\mathcal{C}}(\mathcal{C}/\mathcal{I},M\circ \pi)=0$ for $n>0$. By Proposition 3.4 in \cite{Juan1} we have that $i_{!}$ is fully faithful and we are done. 

\item[(2)] $\Rightarrow)$ Since $i_{\ast}$ satisfies the recollement setup, then it has a right adjoint $i^!$ such that $i^!i_{\ast}\cong 1_{D^b(\mathrm{Mod}(\mathcal{C}/\mathcal{I}))}$, so that $(a)$ and $(b)$ of $(1)$ hold.  Let $M\in\mathrm{Mod}(\mathcal{C})$ and $P^{\bullet}$ be a projective resolution of $M$. Let us check that $\mathbb{TOR}_n^{\mathcal{C}}(\mathcal{C}/\mathcal{I},M)=0$ for $n>>0$. First observe that for $M\in\mathrm{Mod}(\mathcal{C})$ and $I\in\mathrm{Mod}(\mathcal{C}/\mathcal{I})$ injective, we have \[D^b(\mathrm{Mod}(\mathcal{C}))(M,i_{\ast}I[n])\cong D^b(\mathrm{Mod}(\mathcal{C}/\mathcal{I}))(i^{\ast}M,I[n])\cong 0\] for $n>>0$, since $i^{\ast}M\in D^b(\mathrm{Mod}(\mathcal{C}/\mathcal{I}))$. Also, such $n$, where the isomorphism with $0$ holds, only depends on the amplitud of $i^{\ast}M$. If $T=H^{-n}(\mathcal{C}/\mathcal{I}\boxtimes_{\mathcal{C}}P^{\bullet})\neq 0$, we take its injective envelope $I$. Then there is a chain map $\mathcal{C}/\mathcal{I}\boxtimes_{\mathcal{C}}P^{\bullet}\to I[n]$ induced by the following commutative square
\[\begin{tikzcd}
	{Z^{-n}} & {\mathcal{C}/\mathcal{I}\boxtimes_{\mathcal{C}}P^{-n}} \\
	T & I.
	\arrow[hook, from=1-1, to=1-2]
	\arrow[two heads, from=1-1, to=2-1]
	\arrow[dashed, from=1-2, to=2-2]
	\arrow[hook, from=2-1, to=2-2]
\end{tikzcd}\]
More over, the morphism $\mathcal{C}/\mathcal{I}\boxtimes_{\mathcal{C}}P^{\bullet}\to I[n]$ induces the inclusion $T\to I$ taking $-n$-cohomology. By the previous observation, we have that $\mathbb{TOR}_n^{\mathcal{C}}(\mathcal{C}/\mathcal{I},M)=0$ for $n>>0$. By lemma \ref{Gedrich} we conclude that $\mathbb{TOR}^{\mathcal{C}}_n(\mathcal{C}/\mathcal{I},-)=0$ for $n>>0$.\\
Conversely, assume $(a)$, $(b)$ of $(1)$ hold and $\mathbb{TOR}_n^{\mathcal{C}}(\mathcal{C}/\mathcal{I},-)=0$ for $n>>0$. Then by theorem \ref{recoll2} it suffices in order to obtain recollement to show that $i_{\ast}$ has a left adjoint $i^{\ast}$. Clearly, however, $\mathcal{C}/\mathcal{I}\boxtimes_{\mathcal{C}}^{L}-$ is such a left adjoint and the theorem is proved.  
\end{enumerate}
    
\end{proof}


\begin{coro}[\cite{parshall2}, Corolary 3.3]\label{quotientdim}
 Under the hyphotheses of (a,b) in (1) of \ref{Recollalg}, if\\ $gl.dim(\mathcal{C})<\infty$ then $gl.dim(\mathcal{C}/\mathcal{I})<\infty$.
\end{coro}

\begin{proof}
  Suppose that $gl.dim(\mathcal{C})<\infty$. By theorem \ref{Recollalg}, we have a full embedding $$D^b(\mathrm{Mod}(\mathcal{C}/\mathcal{I}))\to D^b(\mathrm{Mod}(\mathcal{C})).$$ Therefore we have isomorphisms $\mathrm{Ext}^n_{\mathcal{C}/\mathcal{I}}(F,F')\cong\mathrm{Ext}^n_{\mathcal{C}}(F\circ\pi ,F'\circ\pi )=0$ for $F,F'\in\mathrm{Mod}(\mathcal{C}/\mathcal{I})$ and $n>gl.dim(\mathcal{C})$. Therefore $gl.dim(\mathcal{C}/\mathcal{I})<\infty$.
\end{proof}



\begin{lema}[\cite{parshall2}, Proposition 3.6]\label{projectiverecoll}
  Let $\mathcal{I}$ be an ideal of $\mathcal{C}$ such that $\mathcal{I}(C,-)\in\mathrm{Mod}(\mathcal{C})$ is projective for all $C\in\mathcal{C}$ and $\mathbb{HOM}_\mathcal{C}(\mathcal{I}_{\mathcal{C}},F\circ \pi)=0$ for all $F\in\mathrm{Mod}(\mathcal{C}/\mathcal{I})$ free. Then $\mathbb{EXT}^n_{\mathcal{C}}(\mathcal{C}/\mathcal{I}_{\mathcal{C}/\mathcal{I}},F\circ\pi)=0$ for all $n>0$ and all $F\in\mathrm{Mod}(\mathcal{C}/\mathcal{I})$ free.
\end{lema}

\begin{proof}
   Let $F\in\mathrm{Mod}(\mathcal{C}/\mathcal{I})$ free. Applying the long exact sequence of $\mathbb{HOM}_\mathcal{C}(-,F\circ\pi)$ to the short exact sequence in $\mathrm{Mod}(\mathcal{C}^{e})$ \[0\to \mathcal{I}\to \mathcal{C}\to \mathcal{C}/\mathcal{I}\to 0\] we obtain the exact sequences \[0=\mathbb{EXT}_{\mathcal{C}}^{n-1}(\mathcal{I},F\circ\pi)\to\mathbb{EXT}^n _{\mathcal{C}}(\mathcal{C}/\mathcal{I},F\circ\pi)\to \mathbb{EXT}_{\mathcal{C}}^n(\mathcal{C},F\circ\pi)=0\]
   for $n>0$.
\end{proof}

\subsection{Finding the left side of a Recollement.}

This results are generalizations of those from \cite{parshall1}.

\begin{obs}\label{obs1}
    Let $\mathcal{B}$ be an abelian subcategory of a triangulated category $\mathcal{D}$. Let $\mathcal{D}_{\mathcal{B}}$ be the full triangulated subcategory of $\mathcal{D}$ generated by $\mathcal{B}$. Assume that we have the commutative triangle
    \[\begin{tikzcd}
	{D^b(\mathcal{B})} \\
	{\mathcal{B}} & {\mathcal{D}_{\mathcal{B}}}
	\arrow["{i_{\ast}}", from=1-1, to=2-2]
	\arrow[hook, from=2-1, to=1-1]
	\arrow[hook, from=2-1, to=2-2]
\end{tikzcd}\]
where $i_{\ast}$ is a full embedding. Then $i_{\ast}:D^b(\mathcal{B})\to \mathcal{D}_{\mathcal{B}}$ is an equivalence of categories.
\end{obs}

\begin{defi}
   Let $\mathcal{A}$ be an abelian category and $\mathcal{B}$ be a non-empty subcategory of $\mathcal{A}$. We say that $\mathcal{B}$ is a \textbf{thick} subcategory of $\mathcal{A}$ if for an exact sequence in $\mathcal{A}$
   \[0\to X\to Y\to Z\to 0\]
  $Y\in \mathcal{B}$ if and only if $X, Z\in \mathcal{B}$.
\end{defi}

\begin{lema}\label{relative}
   Under the hyphotheses of (a,b) in (1) of \ref{Recollalg}, we have 
\begin{enumerate}
    \item[(1)] $\mathrm{Ann}_{\mathcal{C}}(\mathcal{I})$ is a thick subcategory of $\mathrm{Mod}(\mathcal{C})$,
   \item[(2)] The full subcategory of $D^b(\mathrm{Mod}(\mathcal{C}))$, $D^b_{\mathrm{Ann}_{\mathcal{C}}(\mathcal{I})}(\mathrm{Mod}(\mathcal{C}))$, consisting of the complexes with cohomology in $\mathrm{Ann}_{\mathcal{C}}(\mathcal{I})$ is triangulated and is generated by $\mathrm{Ann}_{\mathcal{C}}(\mathcal{I})$.
   \item[(3)] There is an equivalence of categories $D^b_{\mathrm{Ann}_{\mathcal{C}}(\mathcal{I})}(\mathrm{Mod}(\mathcal{C}))\simeq D^b(\mathrm{Mod}(\mathcal{C}/\mathcal{I}))$.
\end{enumerate}
\end{lema}
\begin{proof}
\begin{enumerate}
       \item[(1)] By theorem \ref{Recollalg}, the functor $i_{\ast}:D^b(\mathrm{Mod}(\mathcal{C}/\mathcal{I}))\to D^b(\mathrm{Mod}(\mathcal{C}))$ induced by the functor $\pi_{\ast}:\mathrm{Mod}(\mathcal{C}/\mathcal{I})\to \mathrm{Mod}(\mathcal{C})$ is a full embedding. Let $\eta: 0\to X\to Y\to Z\to 0$ be an exact sequence in $\mathrm{Mod}(\mathcal{C})$. If $X,Y\in \mathrm{Ann}_{\mathcal{C}}(\mathcal{I})$, then $\eta\in \mathrm{Ext}^1_{\mathcal{C}}(Z,X)\cong \mathrm{Ext}^1_{\mathcal{C}/\mathcal{I}}(Z,X)$, so there exist $\epsilon: 0\to X\to Y'\to Z\to 0$ such that $\eta\cong \epsilon $ and $Y'\in \mathrm{Ann}_{\mathcal{C}}(\mathcal{I})$. Therefore $Y\cong Y'\in \mathrm{Ann}_{\mathcal{C}}(\mathcal{I})$. If $Y\in\mathrm{Ann}_{\mathcal{C}}(\mathcal{I})$, then it is clear that $X, Y\in\mathrm{Ann}_{\mathcal{C}}(\mathcal{I})$. 
    \item[(2)] Let us see that in general $\mathrm{Ann}_{\mathcal{C}}(\mathcal{I})$ generates the full subcategory $D^b_{\mathrm{Ann}_{\mathcal{C}}(\mathcal{I})}(\mathrm{Mod}(\mathcal{C}))$ by induction on the amplitud of non-zero cohomology of the complex $X^{\bullet}$, $n$. \\
     If $n=0$, then $X^{\bullet}\cong M[k]$ in $D^b(\mathrm{Mod}(\mathcal{C}))$ for some $M\in \mathrm{Ann}_{\mathcal{C}}(\mathcal{I})$ and there is nothing to prove. Suppose that the sentence is valid for all $k<n$ and let $(m-n+1,m)$ the amplitud of non-zero cohomology of $X^{\bullet}$. We have the distinguished triangle formed with the canonical truncations
     $$\tau^{\leq m-1}X^{\bullet}\to X^{\bullet}\to \tau^{\geq m}X^{\bullet}\to $$
   Since $\tau^{\leq m-1}X^{\bullet}, \tau^{\geq m}X^{\bullet}$ are generated by $\mathrm{Ann}_{\mathcal{C}}(\mathcal{I})$, then $X^{\bullet}$ is generated by $\mathrm{Ann}_{\mathcal{C}}(\mathcal{I})$.\\
   Let us see that $D^b_{\mathrm{Ann}_{\mathcal{C}}(\mathcal{I})}(\mathrm{Mod}(\mathcal{C}))$ is triangulated. Clearly it is closed by shifts. Let $X^{\bullet}\to Y^{\bullet}\to Z^{\bullet}\to $ be a distinguished triangle in $D^b(\mathrm{Mod}(\mathcal{C}))$ such that two of the tree terms are in $D^b_{\mathrm{Ann}_{\mathcal{C}}(\mathcal{I})}(\mathrm{Mod}(\mathcal{C}))$. By the long exact sequence of cohomology and the fact that $\mathrm{Ann}_{\mathcal{C}}(\mathcal{I})$ is thick, it follows that the third term is also in $D^b_{\mathrm{Ann}_{\mathcal{C}}(\mathcal{I})}(\mathrm{Mod}(\mathcal{C}))$.
   \item[(3)] It follows directly from \ref{obs1}.
   \end{enumerate}
\end{proof}

Let $\mathcal{C'}$ be a full subcategory of $\mathcal{C}$. By proposition 2.3 in \cite{LGtesis}, we have adjoint functors 

\[\begin{tikzcd}
	{\mathrm{Mod}(\mathcal{C})} && {\mathrm{Mod}(\mathcal{C'})}
	\arrow[""{name=0, anchor=center, inner sep=0}, "{res_{\mathcal{C'}}}", from=1-1, to=1-3]
	\arrow[""{name=1, anchor=center, inner sep=0}, "{\mathbb{HOM}_{\mathcal{C'}}(\mathcal{C}_{\mathcal{C}},-)}", shift left=3, curve={height=-12pt}, from=1-3, to=1-1]
	\arrow[""{name=2, anchor=center, inner sep=0}, "{_{\mathcal{C}}\mathcal{C}\boxtimes_{\mathcal{C'}}-}"', shift right=3, curve={height=12pt}, from=1-3, to=1-1]
	\arrow["\perp"{pos=0.6}, shift right=3, draw=none, from=0, to=2]
	\arrow["\perp", shift right=3, draw=none, from=1, to=0]
\end{tikzcd}\]

We have a generalization of [\cite{parshall1}, Prop. 2.1]

\begin{prop}\label{prop21}
    Let $\mathcal{C'}$ be a full subcategory of $\mathcal{C}$. Put $\mathcal{D}=D^+(\mathrm{Mod}(\mathcal{C}))$,\quad $\mathcal{D}''=D^+(\mathrm{Mod}(\mathcal{C'}))$, and let $j^{\ast}:\mathcal{D}\to \mathcal{D}''$ be the induced functor by $res_{\mathcal{C'}}$ (respectively $j_{\ast}:\mathcal{D}''\to \mathcal{D}$ the right derived functor of $\mathbb{HOM}_{\mathcal{C'}}(\mathcal{C}.-)$). Then $(j^{\ast},j_{\ast})$ is an adjoint pair  such that the counit $j^{\ast}j_{\ast}\to 1_{\mathcal{D}''}$ is an isomorphism. In particular, $j_{\ast}$ is a full embedding, so that $\mathrm{Ext}^n_{\mathcal{C'}}(M,N)\cong \mathcal{D}(j_{\ast}M,j_{\ast}N[n])$ for all $n\geq 0$ and all $M,N\in\mathrm{Mod}(\mathcal{C'})$. If $\mathcal{D}$ (resp. $\mathcal{D}''$) is replaced by $D^-(\mathrm{Mod}(\mathcal{C}))$ (resp. $D^-(\mathrm{Mod}(\mathcal{C'}))$) and $j_{\ast}$ by $j_!$ where $j_!$ is the left derived functor of $\mathcal{C}\boxtimes_{\mathcal{C'}}-$, then, in a similar way, $(j_!,j^{\ast})$ is an adjoint pair such that the unit $1_{\mathcal{D}''}\to j^{\ast}j_!$ is an isomorphism (and $j_!$ is a full embedding). 
\end{prop}

\begin{proof}
   Since $res_{\mathcal{C'}}:\mathrm{Mod}(\mathcal{C})\to \mathrm{Mod}(\mathcal{C'})$ is exact, the derived functor $j^{\ast}:\mathcal{D}\to \mathcal{D}''$ is defined by $j^{\ast}(X^{\bullet})^n=X^n|_{res_{\mathcal{C'}}}$ for each complex $X^{\bullet}\in \mathcal{D}$. It is well known that the functor $\mathbb{HOM}_{\mathcal{C'}}(\mathcal{C},-)$ preserves injectives. Since $D^+(\mathrm{Mod}(\mathcal{C'}))$ is equivalent to the homotopy category $K^+(\mathcal{I})$, where $\mathcal{I}$ is the full subcategory of $\mathrm{Mod}(\mathcal{C'})$ consisting of injective modules, the proposition follows directly from the standard identity:
 \[res_{\mathcal{C'}}(\mathbb{HOM}_{\mathcal{C'}}(\mathcal{C},-))= 1_{\mathrm{Mod}(\mathcal{C'})}\]
 Indeed, let $I^{\bullet}\in K^+(\mathcal{I})$, then \begin{align*}
 j^{\ast}j_{\ast}(I^{\bullet})=&j^{\ast}(\mathbb{HOM}_{\mathcal{C'}}(\mathcal{C},I^{\bullet}))\\
  =&res_{\mathcal{C'}}(\mathbb{HOM}_{\mathcal{C'}}(\mathcal{C},I^{\bullet}))\\
  = & I^{\bullet}
 \end{align*}
For the $j_!$ statement, it is needed the other canonical identity
 \[res_{\mathcal{C'}}(\mathcal{C}\boxtimes_{\mathcal{C'}}-)=1_{\mathrm{Mod}(\mathcal{C'})}.\]
\end{proof}

\begin{defi}
    Let $\mathcal{C}'$ be a full subcategory of $\mathcal{C}$ closed by finite coproducts. We define the ideal $\mathcal{I}_{\mathcal{C}'}$ which consists of morphisms that factors through an object of $\mathcal{C}'$.
\end{defi}

\begin{lema}\label{resann}
    Let $\mathcal{C}'$ be a full subcategory of $\mathcal{C}$ closed by finite coproducts. Then \[\mathrm{Ann}_{\mathcal{C}}(\mathcal{I}_{\mathcal{C}'})=\{M\in\mathrm{Mod}(\mathcal{C}):res_{\mathcal{C}'}(M)=0\}.\]
\end{lema}

\begin{proof}
    Let $M\in\mathrm{Ann}_{\mathcal{C}}(\mathcal{I}_{\mathcal{C}'})$ and $C\in\mathcal{C}'$. Then $$\displaystyle{M(C)\leq_k\sum_{\substack{h\in\mathcal{I}_{\mathcal{C}'}(D,C)\\D\in\mathcal{C}}}Im(h)=\mathcal{I}_{\mathcal{C}'}M(C)=0}.$$ Now, let $M\in\mathrm{Mod}(\mathcal{C})$ such that $res_{\mathcal{C}'}(M)=0$ and take $C\in \mathcal{C}$. For a morphism $D\to X\to C$ with $X\in\mathcal{C}'$, we have $M(D)\to M(X)\to M(C)$ is the zero morphism since $M(X)=0$. Then $\mathcal{I}_{\mathcal{C}'}M(C)=0$.
\end{proof}

The following is a generalization of [\cite{parshall1}, Theorem 2.3]

\begin{teo}\label{idempotent}
  Let $\mathcal{C}'$ be a full subcategory of $\mathcal{C}$ closed by finite coproducts such that $pd_{\mathcal{C}'}(\mathcal{C})<\infty$ and $pd(\mathcal{C}_{\mathcal{C}'})<\infty$. Set $\mathcal{D}=D^b(\mathrm{Mod}(\mathcal{C}))$ and $\mathcal{D}''=D^b(\mathrm{Mod}(\mathcal{C}'))$. Then the functors $j_!,j^{\ast},j_{\ast}$ of \ref{prop21} on the module categories $\mathrm{Mod}(\mathcal{C})$ and $\mathrm{Mod}(\mathcal{C}')$ can be restricted to adjoint functors $(j_!,j^{\ast},j_{\ast})$ on $\mathcal{D}$ and $\mathcal{D}''$ which satisfy the hyphotesis of \ref{recoll1}, thus yielding a recollement in which $i_{\ast}\mathcal{D}'$ identifies with the relative derived subcategory
 $D^b_{\mathrm{Ann}_{\mathcal{C}}(\mathcal{I}_{\mathcal{C}'})}(\mathrm{Mod}(\mathcal{C}))$ of $\mathcal{D}$.
\end{teo}

\begin{proof}
   Let us see that the functor $j_{\ast}=\mathbb{HOM}_{\mathcal{C}'}(\mathcal{C},-)$ is right cohomologically finite. Let $M\in \mathrm{Mod}(\mathcal{C}')$. Since $pd(_{\mathcal{C}'}\mathcal{C})<\infty$ we have that 
   $\mathbb{EXT}^n_{\mathcal{C}'}(_{\mathcal{C}'}\mathcal{C},M)
   = 0$ for $n>>0$.
   Analogously, since $pd\mathcal{C}_{\mathcal{C}'}<\infty$, the functor $j_{\ast}=\mathcal{C}\boxtimes_{\mathcal{C}'}-$ is left cohomologically finite.
   
   Therefore $j_!,j_{\ast}:D^b(\mathrm{Mod}(\mathcal{C}'))\to D^b(\mathrm{Mod}(\mathcal{C}))$. Thus, by \ref{prop21}, the triple $(j_!,j^{\ast},j_{\ast})$ as functors satisfy the conditions of theorem \ref{recoll1}. It remains to identify $\mathcal{D}'$. Since a complex in $D^b(\mathrm{Mod}(\mathcal{C}'))$ is isomorphic to the zero complex if and only if it is acyclic, it follows from the exactness of the functor $j^{\ast}$ that a complex $X^{\bullet}\in D^b(\mathrm{Mod}(\mathcal{C}))$ satisfies $j^{\ast}(X^{\bullet})\cong 0$ if and only if $X^{\bullet}$ has cohomology groups satisfying $res_{\mathcal{C}'}(H^n(X^{\bullet}))=0$, that is (lemma \ref{resann}), if and only if $X^{\bullet}$ belongs to the relative derived category $D^b_{\mathrm{Ann}_{\mathcal{C}}(\mathcal{I}_{\mathcal{C}'})}(\mathrm{Mod}(\mathcal{C}))$.
\end{proof}

\[
\]


\section{Singularity category}

The singularity category $\mathbf{D}_{sg}(A)$ of an algebra A over a field $K$, introduced by
R.O. Buchweitz in \cite{Buchweitz}, is defined as the Verdier quotient 
$$\mathbf{D}_{sg}(A)=\mathbf{D}^{b}(\mathrm{mod}(A))/\mathrm{perf}(A)$$ of the bounded derived category $\mathbf{D}^{b}(\mathrm{mod}(A))$ by the category of perfect complexes.
In recent years, D. Orlov (\cite{Orlov1} ) rediscovered the notion of singularity categories in his
study of B-branes on Landau-Ginzburg models in the framework of the Homological
Mirror Symmetry Conjecture. The singularity category measures the homological
singularity of an algebra in the sense that an algebra A has finite global dimension
if and only if its singularity category $\mathbf{D}_{sg}(A)$ vanishes.\\
Let $\mathcal{A}$  be an abelian category with enough projective objects. We denote by $\mathrm{Perf}(\mathcal{A})$ the full subcategory of $\mathbf{D}^{b}(\mathcal{A})$ consisting of complexes isomorphic  in $\mathbf{D}^{b}(\mathcal{A})$  to a bounded complex $P^{\bullet}$ of projective objects of $\mathcal{A}$. It is easy to see that  $\mathrm{Perf}(\mathcal{A})$ is a thick triangulated subcategory of $\mathbf{D}^{b}(\mathcal{A})$.

\begin{defi}$\textnormal{\cite[Definition  in p. 3768]{Zhao}}$
Let $\mathcal{A}$  be an abelian category with enough projective objects.
The singularity category of $\mathcal{A}$ is defined to be the
following Verdier quotient triangulated category
$$\mathbf{D}_{sg}(\mathcal{A})=\mathbf{D}^{b}(\mathcal{A})/\mathrm{Perf}(\mathcal{A}).$$
\end{defi}
For the construction of the Verdier's quotient see for example \cite{Thorsten} or \cite{Neeman}.


\begin{obs} Let $\Lambda$ be a ring. It is importan to consider $\mathbf{D}_{sg}(\mathcal{A})$ where $\mathcal{A}=\mathrm{Mod}(\Lambda)$ instead of just $\mathcal{A}=\mathrm{mod}(\Lambda)$  (the category of finitely generated $\Lambda$-modules). Because $\mathbf{D}_{sg}(\mathrm{Mod}(\Lambda))$  is the category that measures de singularity of $\Lambda$ in the sense that $\mathbf{D}_{sg}(\mathrm{Mod}(\Lambda))=0$ if and only if $\mathrm{gl.dim}(\Lambda)<\infty$, for any ring $\Lambda$ (see Remark 6.9 in 
\cite{Panagiotis}.) 
\end{obs}

\subsection{Singularity Category of a Triangular Matrix Category.}

In this part we give generalizations of some results of Pin Liu and Ming Lu in \cite{liu}. For the following  $\mathcal{U}$ and $\mathcal{T}$ will be aditive $k$-categories, $M\in\mathrm{Mod}(\mathcal{T}^{op}\otimes_k\mathcal{U})$ and we consider the triangular matrix category $\Lambda:=\left[\begin{smallmatrix}
       \mathcal{T}&0\\M&\mathcal{U}
   \end{smallmatrix}\right]$
 constructed in \cite{LeOS} and defined as follows.
\begin{defi}$\textnormal{\cite[Definition 3.5]{LeOS}}$ \label{defitrinagularmat}
Let $\mathcal{U}$ and $\mathcal{T}$ be two $k$-categories, and $M\in\mathrm{Mod}(\mathcal{T}^{op}\otimes_k\mathcal{U})$.
The \textbf{triangular matrix category}
$\Lambda=\left[ \begin{smallmatrix}
\mathcal{T} & 0 \\ M & \mathcal{U}
\end{smallmatrix}\right]$ is defined as below.
\begin{enumerate}
\item [(a)] The class of objects of this category are matrices $ \left[
\begin{smallmatrix}
T & 0 \\ M & U
\end{smallmatrix}\right]  $ with $ T\in \mathcal{T} $ and $ U\in \mathcal{U} $.

\item [(b)] For objects
$\left[ \begin{smallmatrix}
T & 0 \\
M & U
\end{smallmatrix} \right] ,  \left[ \begin{smallmatrix}
T' & 0 \\
M & U'
\end{smallmatrix} \right]\in\Lambda$, we define $$\Lambda\left (\left[ \begin{smallmatrix}
T & 0 \\
M & U
\end{smallmatrix} \right] ,  \left[ \begin{smallmatrix}
T' & 0 \\
M & U'
\end{smallmatrix} \right]  \right)  := \left[ \begin{smallmatrix}
\mathcal{T}(T,T') & 0 \\
M(T,U') & \mathcal{U}(U,U')
\end{smallmatrix} \right].$$
\end{enumerate}
The composition is given by
\begin{eqnarray*}
\circ&:&\left[  \begin{smallmatrix}
{\mathcal{T}}(T',T'') & 0 \\
M(T',U'') & {\mathcal{U}}(U',U'')
\end{smallmatrix}  \right] \times \left[
\begin{smallmatrix}
{\mathcal{T}}(T,T') & 0 \\
M(T,U') & {\mathcal{U}}(U,U')
\end{smallmatrix} \right]\longrightarrow\left[
\begin{smallmatrix}
{\mathcal{T}}(T,T'') & 0 \\
M(T,U'') & {\mathcal{U}}(U,U'')\end{smallmatrix} \right] \\
&& \left( \left[ \begin{smallmatrix}
t_{2} & 0 \\
m_{2} & u_{2}
\end{smallmatrix} \right], \left[
\begin{smallmatrix}
t_{1} & 0 \\
m_{1} & u_{1}
\end{smallmatrix} \right]\right)\longmapsto\left[
\begin{smallmatrix}
t_{2}\circ t_{1} & 0 \\
m_{2}\bullet t_{1}+u_{2}\bullet m_{1} & u_{2}\circ u_{1}
\end{smallmatrix} \right].
\end{eqnarray*}
\end{defi}

   We will consider the following subcategories of $\Lambda$; $\mathcal{U}':=\left[\begin{smallmatrix}
       0&0\\M&\mathcal{U}
   \end{smallmatrix}\right]$ and $\mathcal{T}':=\left[\begin{smallmatrix}
       \mathcal{T}&0\\M&0
   \end{smallmatrix}\right]$. Note that $\mathcal{U}'\cong \mathcal{U}$ and $\mathcal{T}'\cong \mathcal{T}$.

   \begin{prop}
       The following statements hold
       \begin{enumerate}
           \item $\mathrm{Mod}(\Lambda)(\mathcal{U}',\mathcal{T}')=0,$
           \item For $\left[\begin{smallmatrix}
       T&0\\M&U
   \end{smallmatrix}\right], \left[\begin{smallmatrix}
       T'&0\\M&U'
   \end{smallmatrix}\right]\in\Lambda$, we have $\left[\begin{smallmatrix}
       T&0\\M&U
   \end{smallmatrix}\right]\amalg \left[\begin{smallmatrix}
       T'&0\\M&U'
   \end{smallmatrix}\right]=\left[\begin{smallmatrix}
       T\amalg T'&0\\M&U\amalg U'
   \end{smallmatrix}\right]$. In particular $\mathcal{T}'$ and $\mathcal{U}'$ are closed by finite coproducts.
   \item $\Lambda/\mathcal{I}_{\mathcal{T}'}\simeq \mathcal{U}'$.
       \end{enumerate}
   \end{prop}

   The following lemma is a generalization of [\cite{Chen}, Theorem 2]

 \begin{lema}
   If $pd(M_{\mathcal{T}})<\infty$, then there exists a recollement of bounded derived categories
\[\begin{tikzcd}
	{D^b(\mathrm{Mod}(\Lambda/\mathcal{I}_{\mathcal{T'}}))} && {D^b(\mathrm{Mod}(\Lambda))} && {D^b(\mathrm{Mod}(\mathcal{T}'))}
	\arrow["{i_{\ast}=i_{!}}", from=1-1, to=1-3]
	\arrow["{i^{!}}", shift left=3, curve={height=-12pt}, from=1-3, to=1-1]
	\arrow["{i^{\ast}}"', shift right=3, curve={height=12pt}, from=1-3, to=1-1]
	\arrow["{j^{\ast}=j^{!}}", from=1-3, to=1-5]
	\arrow["{j_{!}}"', shift right=3, curve={height=12pt}, from=1-5, to=1-3]
	\arrow["{j_{\ast}}", shift left=3, curve={height=-12pt}, from=1-5, to=1-3]
\end{tikzcd}\]
where
\[
\begin{array}{ll}
 i^{\ast}=\Lambda/\mathcal{I}_{\mathcal{T'}}\boxtimes^L_{\Lambda}-,   & j_!=\Lambda\boxtimes^L_{\mathcal{T}'}-, \\
 i_{\ast}=i_!=_{\Lambda}\Lambda/\mathcal{I}_{\mathcal{T'}}\boxtimes^L_{\Lambda/\mathcal{I}_{\mathcal{T'}}}-,  & j^!=j^{\ast}=_{\mathcal{T}'}\Lambda\boxtimes^{L}_{\Lambda}-, \\
 i^!=\mathbb{RHOM}_{\Lambda}(\Lambda/\mathcal{I}_{\mathcal{T'}},-), & j_{\ast}=\mathbb{RHOM}_{\mathcal{T}'}(\Lambda_{\Lambda},-).
\end{array}
\]
\end{lema}   

\begin{proof}
Let us check that $pd(_{\mathcal{T}'}\Lambda)<\infty$. Let $\mathfrak{M}=\left[\begin{smallmatrix}
       T&0\\M&U
   \end{smallmatrix}\right]$ and $F\in\mathrm{Mod}(\mathcal{T'})$. Then
 \begin{align*}
     \mathbb{EXT}^n_{\mathcal{T'}}(\Lambda_{\Lambda}, F)(\mathfrak{M})=&\mathrm{Ext}^n_{\mathcal{T'}}(\Lambda(\mathfrak{M},-)|_{\mathcal{T}'},F)\\
     =0&
 \end{align*}
for all $n>0$ since $\Lambda(\mathfrak{M},-)|_{\mathcal{T'}}=\mathcal{T'}(\left[\begin{smallmatrix}
       T&0\\M&0
   \end{smallmatrix}\right],-)\in\mathrm{Mod}(\mathcal{T'})$ is projective.

 Now let us see that $pd(\Lambda_{\mathcal{T'}})<\infty$. Let $\mathfrak{M}\in \Lambda$ as before, then
 \begin{align*}
     \Lambda(-,\mathfrak{M})|_{\mathcal{T'}^{op}}\cong &\Lambda(-,\left[\begin{smallmatrix}0&0\\M&U\end{smallmatrix}\right])|_{\mathcal{T'}^{op}}\amalg \mathcal{T'}(-,\left[\begin{smallmatrix}
       T&0\\M&0
   \end{smallmatrix}\right]).
 \end{align*}
 For 
$\left[\begin{smallmatrix}T'&0\\M&0\end{smallmatrix}\right]\in\mathcal{T'}$ we have
     $$\Lambda(\left[\begin{smallmatrix}T'&0\\M&0\end{smallmatrix}\right],\left[\begin{smallmatrix}0&0\\M&U\end{smallmatrix}\right])=\left[\begin{smallmatrix}0&0\\M(T',U)&0\end{smallmatrix}\right]\cong M(T',U)$$

     and then the functor $\Lambda(-,\left[\begin{smallmatrix}0&0\\M&U\end{smallmatrix}\right])|_{\mathcal{T'}^{op}}$ is identified with the functor $M(-,U)\in\mathrm{Mod}(\mathcal{T}^{op})$ via the isomorphism $\mathrm{Mod}(\mathcal{T})\cong \mathrm{Mod}(\mathcal{T}')$. Since $M(-,U)$ have finite projective dimension by hypothesis, then the functor $\Lambda(-,\left[\begin{smallmatrix}0&0\\M&U\end{smallmatrix}\right])|_{\mathcal{T'}}$ also have finite projective dimension in $\mathrm{Mod}(\mathcal{T'}^{op})$. Now, for $F'\in\mathrm{Mod}(\mathcal{T'}^{op})$, we have

\begin{align*}
    \mathbb{EXT}^n_{\mathcal{T'}^{op}}(_{\Lambda}\Lambda,F')(\mathfrak{M})=&\mathrm{Ext}^n_{\mathcal{T'}^{op}}(\Lambda(-,\mathfrak{M})|_{\mathcal{T'}^{op}},F')\\
    =&0
\end{align*}
for $n>>0$ since $pd_{\mathcal{X}}(\Lambda(-,\mathfrak{M})|_{\mathcal{T'}^{op}})<\infty$. Therefore $pd(\Lambda_{\mathcal{T'}})<\infty$.
 
 By \ref{idempotent} we have a recollement
\[\begin{tikzcd}
{D^b_{\mathrm{Ann}(\mathcal{I}_{\mathcal{T'}})}(\mathrm{Mod}(\Lambda))} && {D^b(\mathrm{Mod}(\Lambda))} && {D^b(\mathrm{Mod}(\mathcal{T'}))}
	\arrow[ from=1-1, to=1-3]
	\arrow[ shift left=3, curve={height=-12pt}, from=1-3, to=1-1]
	\arrow[ shift right=3, curve={height=12pt}, from=1-3, to=1-1]
	\arrow["{j^{\ast}=j^{!}}", from=1-3, to=1-5]
	\arrow["{j_{!}}"', shift right=3, curve={height=12pt}, from=1-5, to=1-3]
	\arrow["{j_{\ast}}", shift left=3, curve={height=-12pt}, from=1-5, to=1-3]
\end{tikzcd}\]
Finally, let us see that $pd_{\Lambda}(\Lambda/\mathcal{I}_{\mathcal{T'}})=0$. Let $\pi_{\mathcal{T'}}:\Lambda\to \Lambda/\mathcal{T'}$ be the canonical projection and take $\left[\begin{smallmatrix}T&0\\M&U\end{smallmatrix}\right], \left[\begin{smallmatrix}T'&0\\M&U'\end{smallmatrix}\right]\in\Lambda$. Then we have 
\begin{align*}
    \Lambda/\mathcal{I}_{\mathcal{T'}}(\left[\begin{smallmatrix}T&0\\M&U\end{smallmatrix}\right],\left[\begin{smallmatrix}T'&0\\M&U'\end{smallmatrix}\right])\cong& \Lambda/\mathcal{I}_{\mathcal{T'}}(\left[\begin{smallmatrix}0&0\\M&U\end{smallmatrix}\right],\left[\begin{smallmatrix}0&0\\M&U'\end{smallmatrix}\right])\\
    \cong&\Lambda(\left[\begin{smallmatrix}0&0\\M&U\end{smallmatrix}\right],\left[\begin{smallmatrix}0&0\\M&U'\end{smallmatrix}\right])\\
    \cong &\Lambda(\left[\begin{smallmatrix}0&0\\M&U\end{smallmatrix}\right],\left[\begin{smallmatrix}T'&0\\M&U'\end{smallmatrix}\right]).
\end{align*}
Then $\Lambda/\mathcal{I}_{\mathcal{T'}}(\left[\begin{smallmatrix}T&0\\M&U\end{smallmatrix}\right],-)\circ\pi_{\mathcal{T'}}\cong \Lambda(\left[\begin{smallmatrix}0&0\\M&U\end{smallmatrix}\right],-)$ is projective. Therefore we meet the conditions of \ref{relative} completing the proof.
\end{proof}

The following theorem is a generalization of [\cite{liu}, Theorem 3.2]

\begin{teo}
   Keep the notation as above. If $pd_{\mathcal{U}}(M)<\infty$ and $pd(M_{\mathcal{T}})<\infty$, then $D_{sg}(\mathrm{Mod}(\Lambda))$ admites a recollement relative to $D_{sg}(\mathrm{Mod}(\Lambda/\mathcal{I}_{\mathcal{T'}}))$ and $D_{sg}(\mathrm{Mod}(\mathcal{T}'))$.
\end{teo}

\begin{proof}
Note that $\Lambda/\mathcal{I}_{\mathcal{T'}}\boxtimes_{\Lambda}-$ preserves projective objects since it is left adjoint of an exact functor. Therefore $i^{\ast}(\mathrm{Perf}(\Lambda))\subseteq \mathrm{Perf}(\Lambda/\mathcal{I}_{\mathcal{T'}})$.  Next, for a family of objects $\{\mathfrak{M}_i=\left[\begin{smallmatrix}0&0\\M&U_i\end{smallmatrix}\right]\}_{i\in I}\subseteq \mathrm{Mod}(\Lambda/\mathcal{I}_{\mathcal{T'}})$ we have that
\begin{align*}
   i_{\ast}(\coprod_{i\in I}\mathrm{Mod}(\Lambda/\mathcal{I}_{\mathcal{T'}})(\mathfrak{M}_i,-))=& (\coprod_{i\in I}\mathrm{Mod}(\Lambda/\mathcal{I}_{\mathcal{T'}})(\mathfrak{M}_i,-))\circ \pi_{\mathcal{T'}}\\
   \cong & \coprod_{i\in I}\mathrm{Mod}(\Lambda/\mathcal{I}_{\mathcal{T'}})(\mathfrak{M}_i,-)\circ\pi_{\mathcal{T'}}\\
   \cong& \coprod_{i\in I}\mathrm{Mod}(\Lambda)(\mathfrak{M}_i,-).
\end{align*}
Then $i_{\ast}(\mathrm{Perf}(\Lambda/\mathcal{I}_{\mathcal{T'}}))\subseteq \mathrm{Perf}(\Lambda)$. Let $\phi_{\mathcal{U}'}:\mathrm{Mod}(\Lambda)\to \mathrm{Mod}(\mathcal{U}')$ be the canonical projection.\\

Let us see that $pd_{\Lambda}(F\circ \phi_{\mathcal{U'}})<\infty$ for $F\in\mathrm{Mod}(\Lambda)$ free. Let $F=\coprod_{i\in I}\mathrm{Mod}(\Lambda)(\left[\begin{smallmatrix}T_i&0\\M&U_i\end{smallmatrix}\right],-)$ and $\mathfrak{M}=\left[\begin{smallmatrix}T&0\\M&U\end{smallmatrix}\right]$.

\begin{align*}
    F\circ\phi_{\mathcal{U'}}(\mathfrak{M})=&\coprod_{i\in I} \mathrm{Mod}(\Lambda)(\left[\begin{smallmatrix}T_i&0\\M&U_i\end{smallmatrix}\right], \left[\begin{smallmatrix}0&0\\M&U\end{smallmatrix}\right])\\=&\coprod_{i\in I} \left[\begin{smallmatrix}0&0\\M(T_i,U)&\mathcal{U}(U_i,U)\end{smallmatrix}\right]\\
    \cong & \coprod_{i\in I}M(T_i,U)\amalg\coprod_{i\in I}\mathcal{U}(U_i,U).
\end{align*}
   
Observe that $$\coprod_{i\in I}\mathcal{U}(U_i,U)\cong \coprod_{i\in I} \mathrm{Mod}(\Lambda)(\left(\begin{smallmatrix}0&0\\M&U_i\end{smallmatrix}\right), \left(\begin{smallmatrix}T&0\\M&U\end{smallmatrix}\right))$$ and then the functor $\coprod_{i\in I}\mathcal{U}(U_i,-)\circ \phi_{\mathcal{U}}$ is projective. In the other hand the functor\\ $\coprod_{i\in I}M(T_i,-)\in \mathrm{Mod}(\mathcal{U})$ have finite projective dimension by hyphotesis. Then, via the isomorphism $\Lambda/\mathcal{I}_{\mathcal{T'}}\cong^{\psi} \mathcal{U}$, we have that $\coprod_{i\in I}M(T_i,-)\circ \psi\in \mathrm{Mod}(\Lambda/\mathcal{I}_{\mathcal{T'}})$ have finite projective dimension. Since the functor $(\pi_{\mathcal{T'}})_{\ast}$ is exact and preserves projective objects, we have that $\coprod_{i\in I}M(T_i,-)\circ \psi\circ\pi_{\mathcal{T'}}\in \mathrm{Mod}(\Lambda)$ have finite projective dimension and then $pd_{\Lambda}(F\circ \phi_{\mathcal{U'}})<\infty$.\\
Let $F\in\mathrm{Mod}(\Lambda)$ free and $\mathfrak{M}=\left[\begin{smallmatrix}T&0\\M&U\end{smallmatrix}\right]$. Then
\begin{align*}
    \mathbb{EXT}^{0}_{\Lambda}(\Lambda/\mathcal{I}_{\mathcal{T'}},F)(\mathfrak{M})\cong &
 \mathrm{Ext}^{0}_{\Lambda}(\Lambda/\mathcal{I}_{\mathcal{T'}}(\mathfrak{M},-)\circ\pi_{\mathcal{T'}},F)\\
  =&\mathrm{Mod}(\Lambda)(\Lambda(\left[\begin{smallmatrix}0&0\\M&U\end{smallmatrix}\right],-),F)\\
  \cong & F(\left[\begin{smallmatrix}0&0\\M&U\end{smallmatrix}\right])\\
  =& F\circ \phi_{\mathcal{U'}}(\mathfrak{M}).
\end{align*}

Also, we have that $\mathbb{EXT}^{i}_{\Lambda}(\Lambda/\mathcal{I}_{\mathcal{T'}},F)=0$ for $i>0$ since $\Lambda/\mathcal{I}_{\mathcal{T'}}(\mathfrak{M},-)\circ \pi_{\mathcal{T'}}$ is projective. Therefore $i^{!}(F)\cong F\circ\phi_{\mathcal{U'}}\in\mathrm{Perf}(\Lambda/\mathcal{I}_{\mathcal{T'}})$ and then $i^!(\mathrm{Perf}(\Lambda))\subseteq \mathrm{Perf}(\Lambda/\mathcal{I}_{\mathcal{T'}})$. 
For any $X\in \mathrm{Perf}(\Lambda)$, consider the following triangle in $D^b(\mathrm{Mod}(\Lambda))$
\[i_!i^!X\to X\to j_{\ast}j^{\ast}X\to.\]
Since $i_!i^!X\in \mathrm{Perf}(\Lambda)$, we have $j_{\ast}j^{\ast}X\in \mathrm{Perf}(\Lambda)$ and then we have
\[j_{\ast}j^{\ast}(\mathrm{Perf}(\Lambda))\subseteq \mathrm{Perf}(\Lambda).\]
By \ref{quotientrecoll} we have a recollement of $D_{sg}(\mathrm{Mod}(\Lambda))$ relative to $D_{sg}(\mathrm{Mod}(\Lambda/\mathcal{I}_{\mathcal{T'}}))$ and $D_{sg}(\mathrm{Mod}(\mathcal{T'}))$.
\end{proof}

The following corollary is a generalization of [\cite{liu}, Corollary 3.3] and the main result of this article

\begin{coro}\label{main}
   Keep the notation as above. If $gl.dim\mathcal{T}<\infty$, then we have an equivalence of categories
  \[D_{sg}(\mathrm{Mod}(\Lambda))\simeq D_{sg}(\mathrm{Mod}(\mathcal{U})).\]
\end{coro}

\section{Quivers, path algebras and path categories}  
A quiver $\Delta$ consists of a set of vertices $\Delta_{0}$  and a set of arrows $\Delta_{1}$ which is the disjoint union of sets 
$\Delta(x,y)$, where the elements of $\Delta(x,y)$ are the arrows $\alpha:x\rightarrow y$ from the vertex $x$ to the vertex $y$. Given a quiver $\Delta$, its path category $\mathrm{Pth}\Delta$ has as objects the vertices of $\Delta$  and the morphisms $x\rightarrow y$ are paths  from $x$ to $y$ which are by definition the formal compositions $\alpha_{n}\cdots\alpha_{1}$ where $\alpha_{1}$ starts in $x$, $\alpha_{n}$  ends in $y$ and the end point of $\alpha_{i}$ coincides with the start point of $\alpha_{i+1}$ for all $i\in\{1,\ldots,n-1\}$. The positive integer $n$ is called the length of the path. There is a path $\xi_{x}$ of length $0$ for each vertex to itself. The composition in $\mathrm{Pth}\Delta$ of paths of positive length is just concatenations whereas the $\xi_{x}$ act as identities.

Given a quiver $\Delta$  and a field $k$, an additive  $k$-category $k\Delta$ is associated to $\Delta$ by taking as the indecomposable objects in $k\Delta$  the vertices of $\Delta$ and hence an arbitrary object of $k\Delta $ is a finite coproduct of indecomposable objects. Given $x,y\in\Delta_{0}$  the set of maps from $x$ to $y$ is given by the  $k$-vector space with basis the set of all paths from $x$ to $y$. The composition in $k\Delta$ is of course obtained by $k$-linear extension of the composition in $\mathrm{Pth} \Delta$, that is, the product of two composable paths is defined to be the corresponding composition, the product of two non-composable paths is, by definition, zero. In this way we obtain an associative $k$-algebra which has unit element if and only if $\Delta_{0}$ is finite (the unit element is given by $\sum _{x\in \Delta_{0}}\xi_{x}$).\\
In $k\Delta$, we denote by $k\Delta^{+}$ the ideal generated by all arrows and by $(k\Delta^{+})^{n}$ the ideal generated by all paths of length $\ge n$.\\
Given vertices $x,y\in\Delta_{0}$, a finite linear combination $\sum_{w}\lambda_{w}w$, where $\lambda_{w}\in k$ and $w$ are paths of length $\ge 2$ from $x$ to $y$, is called a relation on $\Delta$. It can be seen that any ideal $I\subset (k\Delta^{+})^{2}$ can be generated, as an ideal, by relations. If  $I$ is generated as an ideal by the set $\{\rho_{i}\mid i\}$ of relations, we write $I=\langle \rho_{i}\mid i\rangle$.\\
Given a quiver $\Delta=(\Delta_{0},\Delta_{1})$, a representation $V=(V_{x},f_{\alpha})$ of $\Delta$ over $k$ is given by vector spaces $V_{x}$ for all $x\in \Delta_{0}$, and linear maps $f_{\alpha}:V_{x}\rightarrow V_{y}$, for any arrow $\alpha:x\rightarrow y$.   The category of representations of $\Delta$ is the category with objects the representations, and a morphism of representations  $h=(h_{x}): V\rightarrow V'$ is given by maps $h_{x}:V_{x}\rightarrow V'_{x}$ $(x\in\Delta_{0})$ such that $h_{y}f_{\alpha}=f_{\alpha'}h_{x}$ for any $\alpha:x\rightarrow y$. The category of representations of $\Delta$ is denoted by $\mathrm{Rep}(\Delta)$.\\
Given a set of relations $\langle\rho_{i}\!\!\mid \! i\rangle$ of $\Delta$,  we denote by $k\Delta/\langle\rho_{i}\!\!\mid \! i\rangle$ the path category given by the quiver $\Delta$ and relations $\rho_{i}$.  The category  of functors $\mathrm{Mod}\Big(k\Delta/\langle \rho_{i}\!\!\mid \! i\rangle\Big):=\Big(k\Delta/\langle \rho_{i}\!\!\mid \! i\rangle, \mathrm{Mod}(k)\Big)$ can be identified with the representations  of $\Delta$ satisfying the relations $\rho_{i}$ which is denoted by $\mathrm{Rep}(\Delta,\{\rho_{i}\!\! \mid \! i\})$,  (see \cite[p. 42]{RingelTame}).\\
\begin{ej}
    Let $k$ be a field and $\Delta$ the infinite discrete quiver indexed with $\mathbb{Z}$. Considere the $k\Delta$-bimodule $M$ given by 
    \[M(i,j)=\begin{cases}
        k&\text{if }j=i,i-1\\
        0&\text{othercase}.
    \end{cases}\]
    The category $\Lambda=\left(\begin{smallmatrix}
       k\Delta&0\\M&k\Delta
   \end{smallmatrix}\right)$ has as quiver:
    
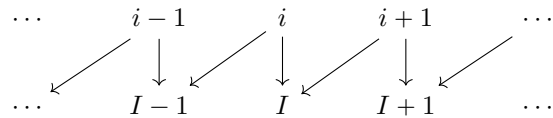
\[\begin{tikzcd}
	\cdots & i-1 & i & i+1 & \cdots \\
	\cdots & I-1 & I & I+1 & \cdots
	\arrow[from=1-2, to=2-1]
	\arrow[from=1-2, to=2-2]
	\arrow[from=1-3, to=2-2]
	\arrow[from=1-3, to=2-3]
	\arrow[from=1-4, to=2-3]
	\arrow[from=1-4, to=2-4]
	\arrow[from=1-5, to=2-4]
\end{tikzcd}\]
Observe that, for each $i\in \Delta$, $M(-,i)=k\Delta(-,i)\amalg k\Delta(-,i+1)$ and $M(i,-)=k\Delta(i,-)\amalg k\Delta(i-1,-)$. Then we have that $pd_{k\Delta}(M)<\infty$ and $pd(M_{k\Delta})<\infty$. Moreover, $gl.dim(k\Delta)=0$, so by corollary \ref{main} we have that $D_{sg}(\mathrm{Mod}(\Lambda))\simeq D_{sg}(\mathrm{Mod}(k\Delta))=0$.

\end{ej}

\appendix

\end{document}